\documentclass[a4 paper, 11pt]{article}
\usepackage{amsthm, a4wide}
\usepackage{amsfonts,bbm}
\usepackage{amsmath}
\usepackage{amssymb}
\usepackage[english]{babel}
\usepackage[active]{srcltx}

\newtheorem{theorem}{Theorem}[section]
\newtheorem{corollary}[theorem]{Corollary}
\newtheorem{lemma}[theorem]{Lemma}
\newtheorem{proposition}[theorem]{Proposition}

\theoremstyle{definition}

\newtheorem{remark}{Remark}

\newcommand{\eps}{\varepsilon}
\let\e=\varepsilon
\newcommand{\ox}{\overline{x}}

\newcommand{\oa}{\overline{\alpha}}

\def\R{\mathbb{R}}

\def\P{\mathbb{P}}
\def\E{\mathbb{E}}

\newcommand{\Si}{\mathbf{S}}
\newcommand{\Mi}{\mathbf{M}}
\newcommand{\tr}{\operatorname{\text{tr}}}
\newcommand{\om}{\Omega} 
\newcommand{\de}{\partial} 
  
\newcommand{\al}{\alpha}

\newcommand{\ovl}{\overline} 
 
\DeclareMathOperator{\dist}{dist}

\begin{document} \title{Nonexistence of nonconstant solutions of  some degenerate Bellman equations and applications to stochastic control}
\author{Martino Bardi,\quad Annalisa Cesaroni,\quad Luca Rossi 
\footnote{Department of 
Mathematics, University of Padova, Via Trieste 63, 35121 Padova, Italy, email: 
bardi@math.unipd.it, annalisa.cesaroni@unipd.it, lucar@math.unipd.it.\newline Current address of A.C. Department of Statistical Sciences,  University of Padova.\newline
M.B. and A.C. were partially supported by   
the European Project Marie Curie ITN ``SADCO - Sensitivity Analysis for 
Deterministic Controller Design''.
A.C. and L.R. were partially supported by the GNAMPA Project 2014 
``Propagation phenomena on lines and networks''.}}

\date{}
\maketitle
\begin{abstract} 
%We derive a Liouville type result for a Bellman operator in a 
%bounded smooth domain. 
For a class of Bellman equations in bounded domains we prove that sub- and supersolutions whose growth at the boundary is suitably controlled %in a suitable way
 must be constant.
The ellipticity of the operator is assumed to degenerate at the boundary and a 
condition %also
 involving also  the 
drift is further imposed.  We apply this result to  stochastic control problems, 
in particular to an exit problem  and to the small discount limit  related with 
ergodic control with state constraints. %problem. 
In this context, our condition on the behavior of the operator near %close to 
the boundary ensures some invariance property of the domain for the associated controlled diffusion process.  
\end{abstract} 
\section {Introduction} 
We consider the fully nonlinear elliptic operator 
%\[
 \begin{equation}
 \label{F} 
F[u]:= \sup_{\alpha\in A} \left(-b(x,\alpha)\cdot Du(x)- \tr 
\left(a(x,\alpha)D^2 u(x)\right)\right) ,
 \end{equation}
 %\]
which is usually called Hamilton-Jacobi-Bellman (briefly, HJB). 
Here ``$\tr$'' denotes the trace of a matrix.
Our first main result is in some sense a counterpart for 
%of Liouville type, although in 
a bounded domain $\Omega\subseteq \R^n$ of results of Liouville type in the whole space.
We assume that $F$
degenerates  at the boundary of $\Omega$, at least  in the normal direction to 
$\partial\Omega$, for some $\oa\in A$,   and that 
the quantity $b(x,\oa)\cdot Dd(x) +\tr(a(x,\oa) D^2d(x))$ is positive near 
$\partial\Omega$, where $d(x)$ is the distance of $x$ from $\partial\Omega$. 
This is related with the viability, or weak invariance, of $\Omega$ for the 
controlled diffusion process associated with the operator $F$.
We show that any viscosity subsolution $u$ of $F[u]=0$ in $\Omega$ such that 
%$u(x)\leq o(-\log(d(x)))$ as $x\to \partial\Omega$
 \begin{equation}
 \label{bcond} 
   u(x)=o(-\log(d(x)))%$
 \quad \text{ as } \; x\to \partial\Omega ,
 \end{equation}
  must be constant. A similar result holds for supersolutions under a stronger 
condition on the coefficients near the boundary related to the invariance of 
the associated diffusion process for all controls. 

Results of nonexistence of nonconstant solutions in  unbounded domains for elliptic equations  are usually called Liouville-type 
theorems. For fully nonlinear %elliptic
 equations they have been studied in the last ten years by several authors, see, e.g., \cite{CDC}, \cite{CF}, and the references therein. They are very different from our result for various reasons. 
In particular, in our case the drift term compensates the degenerate ellipticity of $F$ at the boundary, whereas  in the unbounded case with uniformly elliptic $F$ the drift term represents a difficulty.

%Even though they are called Liouville-type results, the spirit is completely 
%different than in the case of unbounded domains, where, roughly speaking, the 
%drift term represents a complication rather than an advantage (see, e.g., 
%\cite{CDC}).

The main application of our %the
 Liouville-type results concerns the well-posedness of the so-called ergodic 
Hamilton-Jacobi-Bellman  equation
 \begin{equation}
 \label{cell}   
\sup_{\alpha\in A} \left(-b(x,\alpha)\cdot D\chi(x)- \tr (a(x,\alpha) 
D^2\chi(x))-l(x,\alpha)
\right)=c,\qquad   x\in\Omega.
\end{equation}
The unknowns here are $(c,\chi) \in\R\times C(\Omega)$. Under mild continuity 
and boundedness assumptions on the data and the non-degeneracy condition in the 
interior of $\Omega$
\begin{equation}
\label{sell-0}
 \text{$a(x,\alpha)>0$  for 
every $x\in\Omega$ and $\alpha\in 
A$, }
\end{equation}
we prove that this problem has a solution, the additive eigenvalue $c$ is 
unique, and the viscosity solution $\chi$ is unique up to the addition of 
constants among functions $u$
%such that $u(x)\leq -\log(d(x)) o(1)$ as $x\to \partial\Omega$.
satisfying the boundary condition \eqref{bcond}.
% \begin{equation}
% \label{bcond} 
%   u(x)=o(-\log(d(x)))%$
% \quad \text{ as } \; x\to \partial\Omega.
% \end{equation}

There are several motivations for the above mentioned results, especially 
from stochastic optimal control.
Here we give two applications, in Sections \ref{sec:exit} and \ref{stoch_inter}, 
respectively. The first concerns an exit time problem for the controlled 
diffusion 
\begin{equation}
\label{sys}
\begin{cases} dX_t^{\alpha_\cdot}= b(X_t^{\alpha_\cdot}, \alpha_t)dt+\sqrt{2}\sigma(X_t^{\alpha_\cdot}, 
\alpha_t)dW_t\\ 
X^{\alpha_\cdot}_0=x\in\overline{\Omega},\end{cases} 
\end{equation} 
where $\alpha_\cdot\in\mathcal{A}$ is the control. There is a well-known link between 
the associated value function $v$ and the HJB operator $F$ given by \eqref{F} 
with $a:=2\sigma\sigma^T$ (see, e.g., \cite{FS}). Using 
our Liouville-type result we show that $v$ is a constant that, under suitable 
conditions, we can explicitly compute.

%is to the so-called ergodic 
%Hamilton-Jacobi-Bellman (briefly, HJB)  equation %control problem
% in bounded domains. 
%To solve this problem one is led to find a pair $(c,\chi)
%\in\R\times C(\Omega)$ satisfying

The second application concerns the so-called small discount limit. %We consider
Letting $(v_\lambda)_{\lambda>0}$ be the infinite-horizon discounted value 
functions:
\begin{equation*}
% \label{vl-0}
v_\lambda (x):=  \inf_{\alpha_\cdot\in\mathcal{A}} \E 
\left[\int_0^\infty e^{-\lambda t} l(X_t^{\alpha_\cdot},\alpha_t) dt \right] , \quad 
x\in\overline\Omega,
\end{equation*} 
where $X_t^{\alpha_\cdot}$ solves \eqref{sys} and $l$ is bounded, 
we prove that $\lambda v_\lambda (x)\to c$ 
%as $\lambda\to 0$
 and %chi(x)=\lim_{\lambda\to 0+} \left( 
$v_\lambda(x) - v_\lambda(\tilde x)\to \chi(x)$ as $\lambda\to 0$ locally 
uniformly, where $(c, \chi)$ %is the constant solving
solves  the ergodic HJB equation \eqref{cell} and $\tilde x$ is any point in 
$\Omega$. This is related with the ergodic control problem for stochastic 
processes, where $c$ is the optimal cost in the minimisation of 
the limit of $T^{-1} \E 
\big[\int_0^T  l(X_t^{\alpha_\cdot},\alpha_t) dt \big]$ as $T\to+\infty$, and 
%the solution $D\chi$ 
 $D\chi$ allows to synthesise an optimal feedback (at least in principle, under further assumptions). 
There is a large literature on this topic,
see \cite{pll:85, bf2, bbud, bdl} for diffusions reflected at the boundary and 
Neumann boundary conditions in \eqref{cell}, \cite{al, abmem} for periodic 
boundary conditions, 
%\cite{p} for Dirichlet boundary conditions, 
and the recent monograph \cite{abg}, 
as well as the references therein. Some of the cited papers deal with the model problem
\begin{equation*}
\label{model}
-\Delta u+ |Du|^p -f(x) = c ,\qquad \text{in } \Omega ,
\end{equation*} %$$
for $p>1$, which is a special case of \eqref{cell} with unbounded drift 
$b$ and running cost $l$, see, e.g., \cite{pll:85}. %, bf2}. 
Lasry and Lions \cite{ll}  % case of \eqref{cell} 
 showed that, in the case $1<p\leq2$, the problem is uniquely solvable, up to 
constants, under the boundary condition
$$ u(x)\to+\infty%$
 \quad \text{ as }\; %$
 x\to\partial\Omega. 
 $$ They use %, using 
the small discount approximation (see also \cite{p}), and show the connection of the singular boundary condition with stochastic control under state constraints. 

A related result on the small discount approximation for linear operators with singular drift was %has been
 obtained in \cite{lp}, where it is considered in particular the following problem
\[
\lambda u(x)-\frac{b(x)}{d(x)}\cdot Du(x)-a \Delta u -f(x) = 0 ,\qquad \text{in 
} \Omega ,
\] 
under the assumptions $\lambda>0$,  $b(x)\cdot Dd(x)>a>0$ and $b(x)\cdot 
\tau(x)=0$ for every $x\in\partial \Omega$ and every $\tau(x)$ tangential vector 
to $\partial \Omega$ at $x$. The authors prove   the existence of a  
$C^2(\Omega)\cap W^{1,\infty} (\Omega)$ solution of such problem, which is the 
unique $C^2(\Omega)$ solution  
such that \eqref{bcond} is satisfied (see \cite[Thm 6]{lp}).  

%We make 
Here, differently  from \cite{pll:85, ll, lp}, we make the  assumption that $b(x,\alpha)$ is bounded, so $F[u]$ grows 
at most linearly in $Du$ and is not necessarily coercive. On the other hand we 
make assumptions on the coefficients $b$ and $a$ at the boundary that imply an 
invariance property of the domain such that the control problems of reflected 
diffusion and state constraint are essentially equivalent. 
Our weaker boundary condition \eqref{bcond} %$u(x)\leq -\log(d(x)) o(1)$ as 
% $x\to \partial\Omega$
fits this context. Our assumptions on $b$ and $a$ near 
$\partial\Omega$ are related to the characteristic boundary points for linear 
operators \cite{f} and to the irrelevant points for the generalized Dirichlet 
problem \cite{bb, br}. They also arise in the recent work \cite{bcpr} on the 
generalized principal eigenvalue and the Maximum Principle for degenerate 
operators such as $F$.

Let us mention two additional sources of interest for the ergodic PDE \eqref{cell} that we do not develop here.
The first is the theory of homogenisation and singular perturbations for fully 
nonlinear equations. In that context \eqref{cell} is called the cell problem, 
the constant $c$ allows to define an effective Hamiltonian for the limit PDE, 
and $\chi$ is called the corrector and is a fundamental tool for proving the 
convergence, see, e.g., \cite{abmem} and the references 
therein.
%
%A second  motivation to look at such problems 
%
The second is the asymptotic behavior as $t\to +\infty$ 
of solutions of the evolution equation
 \begin{equation*}
 \label{paraintro}   
 u_t%(x,t)
 +\sup_{\alpha\in A} \left(-b(x,\alpha)\cdot Du%(x,t)
 - \tr (a(x,\alpha) 
D^2u%(x,t)
)-l(x,\alpha)
\right)=0,\quad x\in\om,\ t>0, 
\end{equation*}
with initial datum $u(x,0)=u_0(x)$. 
In several cases, such as periodic boundary conditions, it %can be proved
is known that $u(x,t)/t \to c$ as $t\to +\infty$, where $c$ is the constant 
solving \eqref{cell}, a fact also related to ergodic stochastic control, see, 
e.g., \cite{al, abmem}. 
In some cases it can be proved, more precisely, that  $u(x,t) -ct \to \chi(x)$, 
see, e.g.,  \cite{bpt}. The validity of these results with the boundary 
condition \eqref{bcond} will be the subject of further investigations.

The paper is organized as follows. In Section \ref{main} we list the precise 
assumptions and state the two main results of the paper.  Section 
\ref{sec:Lyapunov} is devoted to the construction of a strict supersolution to 
$F[u]=0$ with suitable boundary behavior, that plays the role of a Lyapunov 
function, first when $\partial \Omega$  is of class $C^2$ and then when it is
nonsmooth under further assumptions on $F$.
In Section \ref{sec:Liouville} we prove the non existence of non trivial sub and supersolutions to   $F[u]=0$  for 
$\partial \Omega$ smooth and nonsmooth.
Section \ref{sec:exit} deals with the stochastic control problem with exit times. 
In Section  \ref{sec:ergo} we first prove the well posedness of the ergodic HJB equation 
\eqref{cell} and then interpret the result within stochastic control theory as a vanishing discount limit.

\section{Main results}
\label{main}
Throughout the paper we will assume, if not otherwise stated, that $\Omega$ be a 
bounded connected open  set in $\R^n$ with $C^2$ boundary. 
Let $d(x)$ be the signed distance function from $\partial\Omega$, i.e. 
\begin{equation}\label{distance} d(x):=\dist(x, 
\R^n\backslash\Omega)-\dist(x, \overline{\Omega}).\end{equation} 
We know, from e.g.~\cite[Lemma 14.16]{gt}, that $d$ is of class $C^2$ in some 
neighborhood $\overline{\om_\delta}$ of the boundary, where, here 
and in the sequel, 
\begin{equation}\label{odelta}\Omega_\delta:=\{x\in\Omega \ |\ 
d(x)<\delta\}.\end{equation} 

We consider the fully nonlinear elliptic operator \eqref{F},
with $A\subseteq \R^m$ closed and
\[b:\overline{\Omega}\times A\to \R^n,\qquad a:\overline{\Omega}\times A\to 
\Mi_{n\times n} 
\]
bounded and continuous, $\Mi_{n\times n}$ being the space of  
$n\times n$ real matrices. 
We further assume that $a(x,\alpha)$ is symmetric and nonnegative definite for  
all $x,\alpha$.  
This implies that  $a \equiv \sigma\sigma^T$
for some    $\sigma: \overline{\Omega}\times A\to \Mi_{n\times r}$,  $r\geq 1$. 
Notice that we could take $\sigma\in \Mi_{n\times  n}$ and  symmetric, but we use this notation since in the following, for 
application to stochastic control problems, the matrix $a$ will be obtained from a non symmetric $\sigma$. 

In the interior of $\om$ the diffusion is assumed to be non-degenerate. For 
some results, this will be required in the weak sense that 
the {\em Strong 
Maximum Principle} holds:\\
%Namely,
{if $u\in USC(\Omega)$ is a viscosity subsolution to $F[u]= 0$ in 
$\Omega$ and 
there exists $x_0\in \Omega$ such that $u(x_0)=\max_{\Omega}u$, then $u$ is 
constant.}

\noindent For a detailed analysis of this property for the HJB operator $F$  
see \cite{bdl01}. 
% In particular, this property 
It is satisfied if for any $x\in\Omega$ 
there exists $\alpha_x\in A$ such that the corresponding linear operator $- 
\tr(a(x,\alpha_x)D^2u)$ 
satisfies H\"ormander hypoellipticity condition.
In particular, it holds if $a(x,\alpha_x)>0$ in the sense of matrices.

The main regularity assumptions on the coefficients with respect to $x$ will be 
the following.  
There exists a modulus of continuity~$\omega$ such that 
\begin{equation}\label{reg22}
\forall x,y\in \overline{\Omega},\ \alpha\in A,\qquad
|b(x,\alpha)-b(y,\alpha)|\leq \omega(|x-y|).\end{equation}A modulus of continuity $\omega$ is a nonnegative function, continuous at $0$ with 
$\omega(0)=0$. 
Moreover, we assume that the square root $\sigma$ of the  matrix $a$ is H\"older or Lipschitz continuous uniformly in $\alpha$, with sufficiently large H\"older exponent,  i.e. 
 \begin{equation}\label{reg2} \exists \beta\in 
\left(1/2,1\right],\ B>0,\qquad\forall x,y\in \overline{\Omega}, \alpha\in 
A,\qquad
 |\sigma(x,\alpha)-\sigma(y,\alpha)|\leq B|x-y|^\beta,\end{equation} 
where, even for matrices, $|\cdot|$ stands for the standard Euclidean norm. 
\begin{remark} Note that we are not assuming Lipschitz continuity of 
the coefficients of the operator $F$,  but just uniform continuity of the drift and H\"older continuity of the 
matrix~$\sigma$, which in particular   assure that the operator $F$ is continuous. 
We will strengthen these assumptions to H\"older continuity in Section 6  and to Lipschitz continuity in  
applications to stochastic control problems. 

The regularity assumption \eqref{reg2} is given directly on the matrix $\sigma$   as it is natural for applications  to stochastic control problems. In any case, we recall some sufficient conditions that can be imposed on the matrix $a$ in order to get that \eqref{reg2} holds on the  symmetric square root~$\sigma$.   Assume that $a(\cdot, \alpha)\in W^{2, p}(\Omega)$ for every $\alpha\in A$ with $p>1$ and that $\|a(\cdot,\alpha)\|_{W^{ 2,p}}\leq C$ for $C$ independent of $\alpha\in A$. Then it is proved in \cite[Thm1]{lv95} that $\sigma(\cdot ,\alpha)\in W^{1,2p}(\Omega)$, and moreover it can be observed  that the $W^{1,2p}$ norm of $\sigma$ is independent of $\alpha$. Using Morrey's inequality we have that if $p=\infty$, then \eqref{reg2} is satisfied with $\beta=1$, and if $p>2n$, then \eqref{reg2} is satisfied with $\beta= 1-\frac{n}{p}$.  
\end{remark}

We will make different kind of assumptions about the behavior of $F$ at 
$\partial\om$. The first and weaker one is the following.
\begin{equation}\label{irrelevant}
\begin{array}{c}
\exists\delta,k>0,\ \gamma<2\beta-1,\
\text{such that, for all }\overline{x}\in\partial \Omega,\text{ there is $\oa\in A$ for which }\\
\begin{cases}
\sigma^T (\ox,\oa) Dd(\ox)=0,\\
\forall x\in\Omega\cap B_\delta(\overline x),\qquad
b(x,\oa)\cdot Dd(x) +\tr(a(x,\oa) D^2d(x))\geq k\, d^\gamma(x), 
\end{cases}
\end{array}
\end{equation} 
where $\beta$ is the exponent in \eqref{reg2}.
The first condition in \eqref{irrelevant} means that at any boundary point, the 
normal is a direction of degeneracy for $F$, at least for some $\alpha$.
The second condition is ensured if at the boundary the normal component of 
the drift points inward and is sufficiently large. We will see that 
\eqref{irrelevant} ensures 
the existence of an appropriate Lyapunov function for the system, playing the same role as the already mentioned
condition in \cite{lp}. 
Notice however that  condition \eqref{irrelevant} does not prevent the function 
$b(\cdot,\oa)\cdot Dd(\cdot)+\tr(a(\cdot,\oa) D^2d(\cdot))$ from vanishing at 
$\ox$, even though, in such case, it cannot be Lipschitz continuous because 
$\gamma<1$.

\begin{remark}\label{SCirr}
A sufficient  condition for \eqref{irrelevant} to hold 
(with $\gamma=0$) is the following: there exists $k>0$ such that 
for all $\overline{x}\in\partial \Omega$ there is $\oa\in A$ such that the first condition in \eqref{irrelevant} holds and 
\begin{equation} \label{irr}   
b(\ox,\oa)\cdot Dd(\ox) +\tr(a(\ox,\oa) D^2d(\ox))>k.
\end{equation} Indeed, from \eqref{irr}, using the hypotheses 
  \eqref{reg22}, \eqref{reg2} together with the $C^2$ regularity of $d$ in a 
neighborhood of the compact set $\partial\Omega$, we deduce the existence 
of $\delta>0$ independent of $\overline x$ such that the second condition in 
\eqref{irrelevant} holds with $\gamma=0$ and $k$ 
replaced by $k/2$. Notice that this argument does not work with $k=0$, even 
though $\partial\Omega$ is compact, because of the first condition in 
\eqref{irrelevant}. We exhibit a counter-example in Remark \ref{rem:irr} below, 
where we do not know if the Liouville property holds.
\end{remark} 

Our first main theorem is the following result on nonexistence of nonconstant subsolutions. 
\begin{theorem}\label{liouville}
Let \eqref{reg22}-\eqref{irrelevant}  hold and let $F$ satisfy the Strong Maximum Principle in $\Omega$. 
If $u\in USC(\Omega)$ is a  viscosity subsolution to  
$F[u]=0$ satisfying  
\begin{equation}\label{sublog}
\limsup_{x\to\partial\om} \frac{u(x)}{-\log 
d(x)}\leq 0,
\end{equation}
then $u$ is constant.  
\end{theorem}
\begin{remark}\upshape 
Observe that Theorem \ref{liouville} does not hold without imposing some  growth 
condition at the boundary. Indeed one can check that, in $\Omega=(-1,1)$, if 
$a(x)\geq d^\ell(x)$ for some $\ell>0$ and $d(x)$ small enough, then the 
function \[u(x)= \left|\int_{0}^{x} e^{d^{-\ell}(s)}ds\right| \] is a 
subsolution to $F[u]=0$ in $\Omega_\delta$, for $\delta$ small enough.  
Setting $u(x)=u(1-\delta)$ outside $\Omega\setminus\Omega_\delta$, 
one gets a subsolution to $F[u]=0$. 

On the other hand, we point out that  the growth condition \eqref{sublog} in 
Theorem \ref{liouville} can be weakened to the following: there exists 
$\kappa>0$ such that   \[ \limsup_{x\to\partial\om}  u(x)d(x)^{\kappa}\leq 0.\]  
For more details we refer to Remark \ref{remcrescita}.  
We decided to work with the more restrictive growth 
condition \eqref{sublog} because this provides us with a more detailed 
information on the boundary behavior of the solution of the ergodic problem, 
c.f.~Theorem 
\ref{cellthm}. 
\end{remark} 

The analogous  result for supersolutions to $F[u]=0$ can be 
obtained under stronger assumptions.
In particular we have to strengthen condition \eqref{irrelevant} to hold for 
every $\alpha\in A$, in the following sense. 
\begin{equation}
\label{invariance}
\begin{array}{c}
\exists\delta,k>0,\ \gamma<2\beta-1,\
\text{such that, for all }\overline{x}\in\partial \Omega\text{ and for all $%\oa
\alpha\in A$,}\\
\begin{cases}
\sigma^T (\ox,%\oa
\alpha) Dd(\ox)=0,\\
\forall x\in\Omega_\delta,\qquad
b(x,%\oa
\alpha)\cdot Dd(x) +\tr(a(x,%\oa
\alpha) D^2d(x))\geq k\, d^\gamma(x), 
\end{cases}
\end{array}
\end{equation} 
Moreover, we assume that the operator $F[u]$ satisfies the Strong Minimum Principle, that is, the operator $-F[-u]$ satisfies the 
Strong Maximum Principle. This property for our HJB operator $F$ was studied in \cite{bdl03}.
A  sufficient condition for it to hold true is that 
\begin{equation}
\label{sell} \text{$a(x,\alpha)>0$  for 
every $x\in\Omega$ and $\alpha\in 
A$. }\end{equation}

Next, we use the results on nonexistence of nontrivial solutions to derive the unique solvability of 
the ergodic problem~\eqref{cell} under the boundary condition \eqref{bcond}.
We strengthen the continuity assumption  on the drift  $b$ of the operator 
\eqref{F} and require the same hypothesis on the source term $l$: 
\begin{equation}\label{reg3} \exists\gamma\in (0,1],\quad
\forall x,y\in \overline{\Omega},\ \alpha\in A,\qquad
|b(x,\alpha)-b(y,\alpha)|,|l(x,\alpha)-l(y,\alpha)|\leq B|x-y|^\gamma.\end{equation}
\begin{theorem}
\label{cellthm} 
Under the assumptions \eqref{reg2}, 
\eqref{invariance}-\eqref{reg3}, there 
exists a unique $c\in\R$ such that the equation 
\eqref{cell}   
  admits a  viscosity 
solution $\chi$ satisfying \begin{equation}\label{chi}
\lim_{x\to\partial\om} \frac{\chi(x)}{-\log 
d(x)}= 0.
\end{equation} Moreover $\chi\in C^{2}(\Omega)$ and is 
unique up to addition 
of constants among all solutions to \eqref{cell} which satisfy \eqref{chi}. 
\end{theorem} 

Actually we prove a stronger uniqueness result for $c$ and $\chi$, see Proposition \ref{uniprop}.  
\section{Lyapunov functions}
\label{sec:Lyapunov}
\subsection{The case of smooth boundary}
\begin{proposition}
\label{liap2}
Under assumptions \eqref{reg2} and \eqref{irrelevant}, for every $M\geq 0$,  
there exists    $\delta>0$ such that 
\[F[-\log(d(x))]>M, \qquad x\in \Omega_\delta,\]
where $\Omega_\delta$ is defined in \eqref{odelta}. 
 \end{proposition}
\begin{proof} Take $\delta,k,\gamma$ from  \eqref{irrelevant}. Up to reducing 
$\delta$ if needed, we can suppose that $d\in 
C^2(\overline{\Omega_\delta})$. In particular, for given $x\in\Omega_\delta$, 
we can consider its unique projection on $\partial\Omega$, i.e., the point 
$\overline{x}\in\partial\Omega$ satisfying $|x-\overline{x}|=d(x)<\delta$.
Let $\oa$ be such that    \eqref{irrelevant} is verified. The first condition in 
\eqref{irrelevant} and property \eqref{reg2} yield  
\[
|\sigma^T(x,\overline{\alpha}) Dd(x)|= |\sigma^T(x,\overline{\alpha}) Dd(\ox)|
%=\big|\left(\sigma^T(x,\overline{\alpha})-\sigma^T(\ox,\oa)\right)Dd(\ox)\big|
\leq 
|\sigma(x,\overline{\alpha})-\sigma(\ox,\oa)|\leq B d^\beta(x).\]  
Therefore we get    
\begin{eqnarray*} F[-\log d(x)]&=&   \sup_{\alpha\in A}  
\left( b(x,\alpha)\cdot \frac{Dd(x)}{d(x)}+ \tr \left(a(x,\alpha)\frac{D^2 
d(x)}{d(x)}\right)-\frac{1}{d^2(x)} |\sigma^T (x,\alpha)  Dd(x) |^2\right) \\  
&\geq &  \frac{1}{d(x)} \left(b(x,\oa)\cdot Dd(x)+ \tr (a(x,\oa)D^2 
d(x)\right)-\frac{1}{d^2(x)}  
| \sigma^T (x,\oa) Dd(x)|^2\\ &\geq & 
kd^{\gamma-1}(x)-B^2d^{2\beta-2}(x).
 \end{eqnarray*}  Since $\gamma-1<2\beta-2$, 
we can further reduce $\delta$, independently of $x$, in such 
a way that 
$kd^{\gamma-1}-B^2d^{2\beta-2}>M$ in $\Omega_\delta$, 
concluding the proof of the proposition.
\end{proof}
If we replace \eqref{irrelevant} with \eqref{invariance}, then  the 
following stronger version of Proposition \ref{liap2} holds.
\begin{proposition}
\label{liap22}
Let  \eqref{reg2} and  \eqref{invariance} hold. Then 
 for every $M\geq 0$ there exists    $\delta>0$ such that    
\begin{equation}\label{invliap}F[\log d(x)]<-M, \qquad 
x\in\Omega_\delta.\end{equation}   
\end{proposition} 
\begin{proof} By explicit computation, 
\[ F[\log d(x)]=  - \inf_{\alpha\in A}  
\left( b(x,\alpha)\cdot \frac{Dd(x)}{d(x)}+ \tr \left(a(x,\alpha)\frac{D^2 
d(x)}{d(x)}\right)-\frac{1}{d^2(x)} |\sigma^T (x,\alpha)  Dd(x) |^2\right). \]
The conclusion then follows from a straightforward adaptation of the argument 
in 
the proof of Proposition~\ref{liap2}.
\end{proof} 

We conclude this section with two comments about the optimality of condition 
\eqref{irrelevant} for the construction of the Lyapunov function.

\begin{remark}\label{rem:irr}
The function $-\log d(x)$ might not be a 
Lyapunov function if \eqref{irrelevant} is  relaxed to the following: for all
$\overline{x}\in\partial \Omega$ there is $\oa\in A$ for which
\begin{equation}\label{irrw} 
\sigma^T (\ox,\oa) Dd(\ox)=0,\qquad
b(\ox,\oa)\cdot Dd(\ox) +\tr(a(\ox,\oa) D^2d(\ox))> 0. \end{equation}  
Let us show this with an example.
Let $\Omega$ be a domain in $\R^2$ coinciding, in a
neighborhood of $(0,0)\in\partial \Omega$, with the half-plane $\{(x,y),\
x\in\R,\ y>0\}$. Take $A=\{1,2\}$, 
\[b(x,y,1)=
(0,1),\qquad \sigma(x,y,1)=\left(\begin{array}{cc} 1 & 0\\ 0 &
x^2+y\end{array}\right),\] 
 \[b(x,y,2)= (0,x^4-y),\qquad \sigma(x,y,2)=\left(\begin{array}{cc} y &
0\\ 0 & y\end{array}\right), \]
and $a=\sigma\sigma^T$. 
Note that in a neighborhood of $(0,0)$, $d(x,y)=y$, whence $Dd(x,y)=(0,1)$ and
$D^2 d(x,y)=0$. 
One can readily check that \eqref{irrw} holds at the points $(x,0)$, with 
$\oa=1$ if $x=0$ and $\oa=2$ otherwise, whereas 
the second condition in \eqref{irrelevant} is not verified at $\overline 
x=(0,0)$ by neither $\oa=1$ nor $\oa=2$. 
Let us check that $-\log d(x)$ is not a supersolution of $F=0$ in a
neighborhood of $(0,0)$. For $(x,y)\in\Omega$ close to $(0,0)$, we have that
\begin{eqnarray*} F[-\log d(x)]&=&   \max_{\alpha\in\{0,1\}}  
\left(
b(x,\alpha)\cdot \frac{Dd(x)}{d(x)}+ \tr\left(a(x,\alpha)\frac{D^2
d(x)}{d(x)}\right)-\frac{1}{d^2(x)} |\sigma^T (x,\alpha)  Dd(x) |^2\right) \\ 
&=&   \max  \left\{
\frac{1}{y}-\frac{(x^2+y)^2}{y^2},\frac{x^4}y-2\right\}.
 \end{eqnarray*} 
Now, if $y=x^4$, then 
$$\frac{x^4}y-2=-1<0,\qquad
\frac{1}{y}-\frac{(x^2+y)^2}{y^2}=\frac{1}{y}-\frac{(\sqrt 
y+y)^2}{y^2}<0.$$
This shows that $F[-\log d(x)]<0$ at some points arbitrarily close to $(0,0)$. 
\end{remark}

\begin{remark} 
Note that $\gamma$ in \eqref{irrelevant} satisfies $\gamma<1$, because 
$\beta\leq1$. If 
we allow $\gamma=1$, then we are only able to obtain a ``weak'' Lyapunov 
function, that is a strict supersolution to the 
equation in a neighborhood of the boundary, but not exploding at the boundary. 
Indeed if we assume  that the second condition in \eqref{irrelevant} holds with 
$\gamma=1$, then 
for every $\eta>0$, there exists $\delta >0$ such that 
\[F[-\log(d(x)+\eta)] > 0, \qquad  
x\in\Omega_\delta.\]  
This follows from the same arguments as in the proof of 
Proposition \ref{liap2}, the final step being now the choice of $\delta>0$ 
such that \[
\forall 
x\in\Omega_\delta,\qquad\frac{k d(x)}{ 
d(x)+\eta}-B^2\frac{d^{2\beta}(x)}{(d(x)+\eta)^{2}}>0.\] 
 \end{remark}
 
\subsection{Case of non-smooth boundary} 

When the domain $\Omega$ is not assumed to be $C^2$, a partial result in the 
direction of  Proposition~\ref{liap2} is obtained under 
the following assumption: for all $\ox\in\partial \Omega$, there exist  $\oa\in A$,
 a neighborhood $U$ of $\ox$ and a constant $k>0$ 
such that
\begin{equation}
\label{d_subsol}
-b(x,\oa)\cdot Dd(x)-\tr a(x,\oa)D^2d(x)\leq -k,\qquad x\in\Omega\cap U
\end{equation}
in viscosity sense, and  
\begin{equation}
\label{sigma_perp}
\forall z\in\partial\Omega\cap U ,\quad\forall\nu\in N(z) ,\qquad
\sigma^T(z,\oa) \,\nu=0 , \quad %\text{ in }
% k>0,
\end{equation}
where $N(z)$ is the interior normal cone to $\Omega$ at $z$, %\in 
or, equivalently, the set of generalized exterior normals % cone
 to $\Omega^c$ as defined in \cite[p.~48]{BCD}. Namely, a unit vector $\nu$ 
belongs to $N(z)$, $z\in\partial\Omega$, if there exists $x\in\Omega$ such that 
$x=z+d(x)\nu$.  
Due to the possible lack of regularity of $d$, condition \eqref{d_subsol} has 
to be understood in the following sense:
\begin{equation}\label{pY}
\forall x\in\Omega\cap U,\quad \forall (p,Y)\in J^{2,+}d(x),\qquad
-b(x,\oa)\cdot p-\tr a(x,\oa)Y\leq -k,
\end{equation}
where \[J^{2,+}d(x):=\left\{(p,Y)\in\R^n\times 
\Si_n \ |\ \limsup_{\overline{\Omega}\ni y\to x} \frac{d(y)-d(x)-p\cdot 
(y-x)-\frac{1}{2}(y-x)Y\cdot (y-x)}{|y-x|^2}\leq 0\right\}.\]
If the boundary is smooth, 
\eqref{d_subsol} 
is equivalent to   \eqref{irr} whereas \eqref{sigma_perp} is slightly stronger 
than the first condition in \eqref{irrelevant}, so, combined, they imply 
condition \eqref{irrelevant} (see Remark \ref{SCirr}).
 
 The normal cone $N(z)$ and the set of projections
 \[
 P(x):=\{z\in\partial\Omega\,:\, |x-z|=d(x)\}
 \]
 allow us to give the following representation. 
\begin{lemma} 
\label{lemma_Dd}
For $x\in\Omega$, let $(p,Y)\in J^{2,+}d(x)$. 
Then $p=\lim_jp_j$ where for all $j$ there exist $\lambda_i\geq 0$, $i=1,...,N$, 
$\sum_{i=1}^N \lambda_i=1$, and $z_i\in P(x)$ such that
\[
p_j=\sum_{i=1}^N \lambda_i \nu_i , \qquad \nu_i\in N(z_i).
\]
\end{lemma} 
\begin{proof} The condition $(p,Y)\in J^{2,+}d(x)$ yields 
$$p\in D^+d(x):=\left\{q\in\R^n\ |\ \limsup_{\overline{\Omega}\ni y\to x} 
\frac{d(y)-d(x)-q\cdot (y-x)}{|y-x|}\leq 0\right\}.$$
Proposition II.2.14 in \cite{BCD} states that the set $D^+d(x)$ coincides with 
the closure of the convex hull of the set
\[
\left\{\frac{x-z}{|x-z|} \,:\, z\in P(x)\right\} .
\]
On the other hand, $z\in P(x)$ implies $x=z+d(x)\frac{x-z}{|x-z|}$ and therefore 
$\frac{x-z}{|x-z|}\in N(z)$.
\end{proof}  

\begin{proposition}
\label{liap3} Let \eqref{reg22}, \eqref{reg2},  \eqref{d_subsol},  \eqref{sigma_perp} hold. Then, for any $M\geq  0$, there exists $\delta>0$ such that 
$-\log(d(x))$ satisfies in viscosity sense
\[F[-\log(d(x)] >M, \qquad x\in 
\Omega_\delta.\]
\end{proposition}

\begin{proof} Note that \eqref{reg22}, \eqref{reg2} assure that the operator $F[u]$ is continuous.  For a fixed $\ox\in\partial\Omega$
take  $\overline{\alpha}\in A$ and the neighborhood $U$ such that 
\eqref{d_subsol} and  \eqref{sigma_perp} hold. For $x\in\Omega\cap U$ consider 
$(p,Y)\in J^{2,+}d(x)$. Take $p_j\to p$ as in Lemma \ref{lemma_Dd}. The 
orthogonality condition \eqref{sigma_perp} yields
\[
\sigma^T(x,\overline{\alpha}) p_j= \sum_{i=1}^N \lambda_i 
\sigma^T(x,\overline{\alpha}) \nu_i=\sum_{i=1}^N \lambda_i 
(\sigma^T(x,\overline{\alpha})-\sigma^T(z_i,\overline{\alpha}))\nu_i. 
\]
Then, by \eqref{reg2},
\[
|\sigma^T(x,\overline{\alpha}) p_j|  \leq \sum_{i=1}^N \lambda_i B 
|x-z_i|^\beta = B d^\beta(x),\]
from which, letting $j\to\infty$, 
\[
|\sigma^T(x,\overline{\alpha}) p|  \leq \ B d^\beta(x) .
\]

Now we want to prove that $F[-\log d(x)]>0$ in viscosity 
sense. Note that if $(q,Z)\in J^{2,-}(-\log d(x))$ then there exist $(p, Y)\in J^{2,+} 
d(x)$ such that  
$q=-\frac{p}{d(x)}$and $Z=-\frac{Y}{d(x)}+\frac{p\otimes p}{d^2(x)}$. 
Using \eqref{pY} we derive
 \[\begin{split} 
 F(x,p,Y) &= \sup_{\alpha\in A}  \left( b(x,\alpha)\cdot \frac{p}{d(x)}+ \tr 
a(x,\alpha)\frac{Y}{d(x)}-\frac{1}{d^2(x)}%\sum_{i=1}^n 
   |\sigma^T(x,\alpha) p |^2\right)\\
& \geq \frac{1}{d(x)} \left(b(x,\oa)\cdot p+ \tr 
(a(x,\oa)Y\right)-\frac{1}{d^2(x)}%\sum_{i=1}^n 
   | \sigma^T(x,\oa) p|^2 \\
&\geq \frac{1}{d(x)}\left(k%\frac{k}{d(x)}-\frac{C}
 -  B^2d^{2\beta-1}(x)%^{2\beta-1}}
 \right).
   \end{split}\]
So, %recalling that $\beta\geq \frac{1}{2}$ and 
for given $M>0$, $\tilde\delta>0$ can be chosen
sufficiently small, depending on $\ox$ (because 
$k$ does) but independent of $x,p,Y$, in such a way that  
 $$F[-\log d(x)] > M, \qquad %$ in viscosity sense in $
 x\in\Omega_{\tilde\delta}\cap U$$
 in viscosity sense. This means that there is a neighborhood $W$ of 
$\partial\Omega$ such that 
 $$F[-\log d(x)] > M, \qquad %$ in viscosity sense in $
 x\in\Omega\cap W,$$
and thus the conclusion of the proposition holds for some 
$\delta>0$ by the compactness of $\partial\Omega$.
 \end{proof}
 \begin{remark}\upshape \label{remsuper} If we assume in Proposition \ref{liap3}  that conditions \eqref{d_subsol} and 
\eqref{sigma_perp} hold for every $\alpha\in A$, then we obtain that 
for any $M\geq  0$, there exists $\delta>0$ such that 
$-\log(d(x))$ satisfies in viscosity sense
\[ \inf_{\alpha\in A} \left(-b(x,\alpha)\cdot D(-\log d(x))- \tr 
\left(a(x,\alpha)D^2 (-\log d(x))\right)\right) >M, \qquad x\in 
\Omega_\delta.\]
\end{remark} 
%-------------------------------------------------------------------------------

\section{Nonexistence of nonconstant solutions}
\label{sec:Liouville}
Using the Lyapunov functions constructed in the previous section we now 
prove that $F[u]=0$ has only trivial sub and supersolutions. 
\begin{proof}[Proof of Teorem \ref{liouville}] 
Let  $\delta$ be from Proposition \ref{liap2} corresponding to $M=0$. We can assume without loss of 
generality that $\delta<1$.
Define $\Omega_\delta$ as in \eqref{odelta} and 
$$K_{\delta}:=\max_{\Omega\backslash\Omega_\delta}\,u.$$ 
For $\e>0$ the function $u(x)-(K_{\delta}-\e\log d(x))$ is negative when 
$d(x)=\delta$ and, by \eqref{sublog}, it goes to $-\infty$ as 
$x\to\partial\Omega$. Suppose by contradiction that it is positive somewhere in 
$\Omega_\delta$. Then there exists $x_0\in\Omega_\delta$ such that
\[u(x_0)-(K_{\delta}-\eps \log d(x_0))=  \max_{x\in 
\Omega_\delta} [u(x)-(K_{\delta}-\e\log d(x))].\]  
Since $u$ is a subsolution to $F[u]=0$, this implies that
\[ \eps F[-\log d(x_0)]\leq 0,\] which is impossible by Proposition \ref{liap2}.
We have therefore shown that
\[
\forall x\in\Omega_\delta,\qquad
u(x)\leq K_{\delta}-\eps \log d(x),\]
whence, letting $\e\to0$, $u\leq K_\delta$ in 
$\Omega_\delta$. This means that $u$ achieves its maximum $K_\delta$ 
inside $\Omega$ and then it is constant since $F$ satisfies the Strong Maximum 
Principle in $\Omega$. 
\end{proof}

In the case of nonsmooth boundary, the same result holds true, under stronger 
regularity assumptions on the coefficients. 
\begin{proposition}
\label{L-nonsmooth}  
Instead of $\partial \Omega$ of class $C^2$ assume that  any $\ox\in\partial \Omega$ has a a neighborhood $U$ such that \eqref{d_subsol} and \eqref{sigma_perp} hold for some $\oa\in A$ and $k>0$.
Suppose also that $b,\sigma$ are 
Lipschitz continuous, i.e., %~assume that 
the continuity modulus in \eqref{reg22} is $\omega(r)=Br$ for some $B>0$ and
%that
 $\beta=1$ in \eqref{reg2}, %.  Moreover assume  \eqref{d_subsol},  \eqref{sigma_perp} and that 
 and $F$ satisfies the Strong Maximum  Principle in $\Omega$.  
If $u\in USC(\Omega)$  is a  viscosity subsolution to  
$F[u]=0$ satisfying  \eqref{sublog}
then $u$ is constant.   
\end{proposition} 
\begin{proof} We use  the same notation as in Theorem \ref{liouville}. Now, 
owing to Proposition \ref{liap3}, $0<\delta<1$ is chosen in such a way that 
$-\log d(x)$ satisfies in viscosity sense $F[-\log d(x)]> 1$ in 
$\Omega_\delta$. Using the doubling variables method 
we 
define, for $\eta>0$ and $x,y\in \Omega_\delta$,
\[\Phi(x,y):= u(x)-K_\delta +\eps \log d(y)- \frac{|x-y|^2}{2\eta}.\] Observe 
that, due to \eqref{sublog}, $\sup_{ 
(\Omega_\delta)^2}\Phi<+\infty$.
Take $(x_\eta, y_\eta)$ small such that 
$\Phi(x_\eta, 
y_\eta)=\sup_{ \Omega_\delta^2}\Phi$.  By \cite[Lemma 3.1]{cil}, $\frac{|x_\eta-y_\eta|^2}{\eta}\to 0$ as $\eta\to 0$, and up to subsequences, $x_\eta, 
y_\eta\to x\in \Omega_\delta$ such that $u(x)-K_\delta+\eps\log d(x)=\sup_{ \Omega_\delta}u-K_\delta+\eps\log d$. 

If   $\Phi(x_\eta, y_\eta)\leq 0$ then $\Phi(x,x)\leq 0$ for $x\in\Omega_\delta$, from which, arguing as in the proof of
Theorem \ref{liouville}, we infer
that $u$ achieves a maximum in $\Omega$ and then it is constant. 

Suppose then that $\Phi(x_\eta, y_\eta)> 0$ for all $\eta>0$. By classical 
argument  (see \cite[Theorem 3.2]{cil}), we get that 
 there exist $X,  Y\in\Si_n$ such that 
$\left(\frac{ x_\eta-y_\eta }{ \eta}, X\right)\in J^{2,+}u(x_\eta)$, 
$\left(\frac{ x_\eta-y_\eta }{ \eta},  Y\right)\in J^{2,-}(-\eps\log d(y_\eta))$ 
and  \begin{equation}\label{pq} pX \cdot p-qY\cdot q\leq \frac{1}{\eta}|p-q|^2\qquad \forall\ p,q\in\R^n.\end{equation}
Let $e_i$, for $i=1, \dots, r$ the $i$-th unit vector. Then  it is easy to check  for every $\alpha\in A$, 
$\tr a(x_\eta,\alpha)X= \sum_{i=1}^r (\sigma(x_\eta,\alpha)e_i)X\cdot (\sigma(x_\eta,\alpha)e_i)$. 
So, by \eqref{pq} applied to $p= \sigma(x_\eta,\alpha)e_i,  q=\sigma(y_\eta,\alpha)e_i$, we get  that   
for every $\alpha\in A$, 
\[\tr a(x_\eta,\alpha)X-\tr a(y_\eta,\alpha)Y \leq \frac{1}{\eta} \sum_{i=1}^r\sum_{j=1}^n|\sigma_{ji}(x_\eta,\alpha)-\sigma_{ji}(y_\eta,\alpha)|^2\leq   B^2\frac{|x_\eta-y_\eta|^2}{\eta}\]
where we used the fact that $\sigma$ is Lischitz continuous.
Using the fact that $b$ is Lipschitz continuous we get 
\[|b(y_\eta,\alpha)-b(x_\eta, \alpha)| \frac{ |x_\eta-y_\eta |}{ \eta}	\leq B\left(\frac{ 
|x_\eta-y_\eta |^2}{ \eta}\right).\]
Using the previous inequalities and the fact that $u$ is a subsolution to $F[u]=0$, we get
\begin{eqnarray*}0& \geq &\eps \sup_{\alpha\in A} \left(-b(x_\eta,\alpha)\cdot 
\frac{ x_\eta-y_\eta }{ \eta}- \tr 
a(x_\eta,\alpha) X\right)
\\ & \geq & \eps \sup_{\alpha\in A} \left(-b(y_\eta,\alpha)\cdot \frac{ 
x_\eta-y_\eta }{ \eta}- \tr 
(a(y_\eta,\alpha) Y)\right)- B\eps\frac{|x_\eta-y_\eta|^2}{\eta}- \eps B^2\left(\frac{ 
|x_\eta-y_\eta |^2}{ \eta}\right) ,\end{eqnarray*} which gives, for $\eta$ sufficiently 
small,  a contradiction to the fact that $F[-\log d(x)]>1$ in viscosity sense 
in $\Omega_\delta$.
\end{proof}

If we strengthen condition \eqref{irrelevant} to condition \eqref{invariance}, and let $F[u]$ satisfies the Strong Minimum Principle in $\Omega$, that is $-F[-u]$ satisfies the Strong Maximum Principle in $\Omega$,  we 
also obtain the statement for supersolutions.
\begin{theorem}\label{liouville2}
Let conditions \eqref{reg22}, \eqref{reg2}, \eqref{invariance} hold and let $F$ 
satisfies the Strong Minimum Principle in $\Omega$. If $v\in LSC(\Omega)$ is a  
viscosity supersolution to  
$F[v]=0$ satisfying  
\begin{equation}\label{superlog}
\liminf_{x\to\partial\om} \frac{v(x)}{- \log d(x)}\geq 0,\end{equation} 
then $v$ is constant.  
\end{theorem}
\begin{proof} Note that $v$ is a subsolution to $-F[-u]=0$. So, we conclude, using the same arguments as in the proof of Theorem 
\ref{liouville},  by substituting  Proposition \ref{liap2} with Proposition \ref{liap22}.\end{proof} 

\begin{remark}\label{remcrescita} 
Theorems \ref{liouville} and \ref{liouville2} can be improved by replacing %substituting
 the growth conditions \eqref{sublog} and \eqref{superlog} with, respectively, 
\[\limsup_{x\to\partial\om}  u(x) d(x)^\kappa \leq 0,  \qquad \text{ and }\qquad\liminf_{x\to\partial\om}  v(x) d(x)^\kappa \geq 0\]
for some $\kappa>0$. 

Indeed, observe that for every $M>0$ and every $\kappa>0$, there exists $\delta>0$ such that  
\begin{eqnarray*} F[d(x)^{-\kappa}] &=&   \frac{\kappa}{d(x)^\kappa}  \sup_{\alpha\in A}  
\left( b(x,\alpha)\cdot \frac{Dd(x)}{d(x)}+ \tr \left(a(x,\alpha)\frac{D^2 
d(x)}{d(x)}\right)-\frac{\kappa+1}{d^2(x)} |\sigma^T (x,\alpha)  Dd(x) |^2\right) \\   &\geq & 
k\kappa d^{\gamma-1-\kappa}(x)-B^2\kappa(\kappa+1)d^{2\beta-2-\kappa}(x) \geq M \end{eqnarray*} 
for every $x\in \Omega_\delta$, by assumption \eqref{irrelevant}. 
Analogoulsy we get that for every $M>0$ and every $\kappa>0$ there exists $\delta>0$ such that $F[-d(x)^{-\kappa}]\leq -M$ for every $x\in\Omega_\delta$. 

So, the same arguments of the proofs of Theorems \ref{liouville} and \ref{liouville2} can be repeated by simply substituting $d(x)^{-\kappa}$ to $-\log(d(x))$. 
\end{remark} 

We state the analogous of Theorem \ref{liouville} in the case of nonsmooth boundary. \begin{proposition}
\label{L-nonsmooth2}  
Instead of $\partial \Omega$ of class $C^2$ assume that  any $\ox\in\partial \Omega$ has a a neighborhood $U$ such that \eqref{d_subsol} and \eqref{sigma_perp} hold for every $\alpha\in A$ and $k>0$.
Suppose also that $b,\sigma$ are 
Lipschitz continuous, i.e., %~assume that 
the continuity modulus in \eqref{reg22} is $\omega(r)=Br$ for some $B>0$ and
%that
 $\beta=1$ in \eqref{reg2},  
 and $F$ satisfies the Strong Minimum  Principle in $\Omega$.  
If $v\in LSC(\Omega)$  is a  viscosity supersolution to  
$F[v]=0$ satisfying  \eqref{superlog}
then $v$ is constant.   
\end{proposition} 
\begin{proof} Observe that $-v$ is a viscosity subsolution to 
\[ \inf_{\alpha\in A} \left(-b(x,\alpha)\cdot Du- \tr 
\left(a(x,\alpha)D^2 u\right)\right) =0 \] and satisfies \eqref{sublog}. Moreover by assumption, the operator 
\[ G[u]=\inf_{\alpha\in A} \left(-b(x,\alpha)\cdot Du- \tr 
\left(a(x,\alpha)D^2 u \right)\right)\] satisfies the Strong Maximum Principle. So, recalling Remark \ref{remsuper}, we can repeat the proof of Proposition \ref{L-nonsmooth} by substituting the operator $F[u]$ with  the operator $G[u]$.
 \end{proof}

\section{An application to %exit
 stochastic control problems with exit times} 
 \label{sec:exit}
  We consider the control system \eqref{sys},
%\[
% \begin{equation}
% \label{sys}
% \begin{cases} dX_t^{\alpha_\cdot}= b(X_t^{\alpha_\cdot}, \alpha_t)dt+\sqrt{2}\sigma(X_t^{\alpha_\cdot}, 
% \alpha_t)dW_t\\ 
% X^{\alpha_\cdot}_0=x\in\overline{\Omega},\end{cases} 
% \end{equation} 
%\]
 where  $W_t$ is a $r-$dimensional standard Brownian motion and the control
$\alpha_\cdot:\R_+\to A$ belongs to the set of {\em admissible controls}
$\mathcal{A}$, namely, %of
 progressively measurable processes, with respect to the 
filtration associated to the Brownian motion.

In order to have existence and uniqueness of solutions to the control system, throughout  this section we will strenghten the regularity assumptions on the coefficients to the following.  
We assume  there exists $B>0$ for which 
\begin{equation}\label{reg22s}
\forall x,y\in \overline{\Omega},\ \alpha\in A,\qquad
|b(x,\alpha)-b(y,\alpha)|\leq B|x-y|,\qquad
 |\sigma(x,\alpha)-\sigma(y,\alpha)|\leq B|x-y|.\end{equation} 

Moreover, to ensure existence of  optimal controls for the optimal control problems we are going to consider, 
we assume that the set $A$ is compact and 
 \begin{equation}\label{compact} \{(b(x,\alpha),a(x,\alpha)) \ |\ \alpha\in A\} 
\ \text{ is convex  for all }\ x\in \overline{\Omega}. \end{equation}

Define  for every
$x\in \Omega$ the exit time from the open set $\Omega$ \begin{equation}\label{ex}\tau^{\alpha_\cdot}_x= \inf\{t\geq 0 \ | \ 
X_t^{\alpha_\cdot}\not\in \Omega\}\in\R \cup \{+\infty\}. \end{equation} 

It has been proved in \cite{bj} that under the previous assumptions,
the set $\overline{\Omega}$ is {\em viable} or {\em weakly invariant} %with respect to the previous
for the  control system \eqref{sys}  in the following sense:
for every $x\in \overline\Omega$ there exists an admissible control $\alpha_\cdot$ such 
that $X_t^{\alpha_\cdot}\in \overline{ \Omega}$  almost surely for all $t\geq 0$.
In the next proposition we prove that actually in our case this result can be improved to get also the viability of the open set $\Omega$. %This
Such a result was %has been
 obtained  in \cite{cdf}
for stochastic systems without control. 
\begin{proposition} \label{viable} Let  \eqref{irrelevant}, 
\eqref{reg22s}, \eqref{compact} hold. 
Then   for every $x\in\Omega$ there exists an admissible control $\alpha_\cdot\in\mathcal{A}$ such 
that $X_t^{\alpha_\cdot}\in \Omega$  almost surely for all $t\geq 0$, i.e. $\tau^{\alpha_\cdot}_x=+\infty$ almost surely. \end{proposition}
\begin{proof}  
Let $V$ be a $C^2$ extension of the function $-\log(d(x))$  to the whole $\om$, which coincides with $-\log d(x)$ in a neighborhood of $\partial\Omega$. We can assume that $V\geq 1$ 
in $\om$ and moreover, by Proposition \ref{liap2},   there exists a constant $C\geq 0$ such that 
\[F[V](x)\geq -C\qquad x\in\Omega.\]  

Define for every $\delta>0$, 
and $x\in\Omega$, the exit time \[\tau^{\delta,\alpha_\cdot}_x= \inf\{t\geq 0 \ | \ 
X_t^{\alpha_\cdot}\not\in  \Omega\setminus\overline{\Omega_\delta}\}\in\R \cup \{+\infty\}. \]

By superoptimality principles for viscosity solutions (see \cite[Corollary 6]{c}), we get
for every $\delta>0$, every $x\in\Omega\setminus\Omega_\delta$ and every $t\geq 0$, recalling that $V\geq 1$,
\begin{equation}\label{sup} V(x)\geq \inf_{\alpha_\cdot\in\mathcal{A}} \E 
\left[V(X_{\tau^{\delta,\alpha_\cdot}_x\wedge 
t}^{\alpha_\cdot})-C(\tau^{\delta,\alpha_\cdot}_x\wedge t)\right]\geq 
\inf_{\alpha_\cdot\in\mathcal{A}}  \left[\int_{\{\omega| 
\tau^{\delta,\alpha_\cdot}_x(\omega)\leq t\}} V(X_{\tau^{\delta,\alpha_\cdot}}^{\alpha_\cdot}(\omega))
d \P(\omega)-C t \right],\end{equation}
where $\tau^{\delta,\alpha_\cdot}_x\wedge t=\min \{\tau^{\delta,\alpha_\cdot}_x,t\}.$ Take $\delta>0$ such that $V(x)=-\log d(x)$ for $x\in \overline{\Omega_\delta}\setminus\partial\Omega$. From \eqref{sup} we get that for every $t\geq 0$ 
\[\inf_{\alpha_\cdot\in\mathcal{A}} \P \left(\omega\ |\ 
\tau^{\delta,\alpha_\cdot}_x(\omega)\leq t\right)\leq \frac{V(x)+Ct}{-\log\delta}.\]
Moreover, since $\inf_{\alpha_\cdot} \P \left(\omega\ |\ 
\tau^{\delta,\alpha_\cdot}_x(\omega)\leq t\right)$ is decreasing as $\delta\to 0$,  
we get that, for all $x\in \Omega$,  \[\inf_\delta\inf_{\alpha_\cdot\in\mathcal{A}} \P 
\left(\omega\ |\ \tau^{\delta,\alpha_\cdot}_x(\omega)\leq t\right)=0.\]
So,    for every $t\geq 0$, 
\begin{equation}\label{cont} \inf_{\alpha_\cdot\in\mathcal{A}}  \P \left(\omega\ |\ 
\tau^{ \alpha_\cdot}_x(\omega)\leq t\right)=0. \end{equation} 
Finally we claim that \eqref{cont} implies that 
\begin{equation}\label{inven}\inf_{\alpha_\cdot\in\mathcal{A}} \P \left(\omega\ |\ 
\tau^{ \alpha_\cdot}_x(\omega)<+\infty\right)=0.\end{equation}  
Let  $h$ be a bounded uniformly continuous function such that $h\equiv 0$ in $\Omega$ and $h>0$ in $\R^n\setminus\Omega$.
Let \begin{equation}\label{w} w(x)= \inf_{\alpha_\cdot\in\mathcal{A}}  \E 
\int_0^{+\infty} h(X_t^{\alpha_\cdot})e^{-\nu t} dt.\end{equation} So $0\leq w\leq 
\frac{\|h\|_\infty}{\nu}$. Moreover, by standard dynamic programming principle 
(see \cite{hl}) for every $t>0$
\[w(x)=\inf_{\alpha_\cdot\in\mathcal{A}}  \E \left(w(X_t^{\alpha_\cdot})e^{-\nu t}+\int_0^t 
h(X_s^{\alpha_\cdot})e^{-\nu s} ds\right).\]
Fix $\eps>0$ and take $t$ such that $\frac{\|h\|_\infty}{\nu} e^{-\nu t}\leq \eps$.  So, for every $x\in \Omega$, due to \eqref{cont}, \[w(x) \leq
\inf_{\alpha_\cdot\in\mathcal{A}} \E  \int_0^t h(X_s^{\alpha_\cdot})e^{-\nu s} ds +\eps\leq 
\eps,\] from which we deduce by arbitrariness of $\eps$,  $w\equiv 0$. 
This implies  \eqref{inven}, recalling that under assumption \eqref{compact}, for every initial data  $ x\in\Omega$ there exists an optimal control
 for the control problem \eqref{w} (see \cite{hl}). \end{proof}

Consider a terminal cost $\phi\in %LS
C(\partial \Omega)$ that the controller pays as the system hits the boundary.
We introduce the cost functional 
\[
G(x,\alpha,\omega)= \begin{cases}  
\phi(X_{\tau^{\alpha_\cdot}_x}^{\alpha_\cdot}(\omega)) &  \tau_x^{\alpha_\cdot}(\omega)<+\infty\\ 0 &  \tau_x^{\alpha_\cdot}(\omega)=+\infty,\end{cases}\]
 and 
define the value function  
\[
v(x)=\inf_{\alpha_\cdot\in\mathcal{A}}  \E\ [G(x,\alpha_\cdot,\omega)].\]

We make a non degeneracy assumption on the system at a minimal point for
$\phi$ which is somehow opposite to condition \eqref{irrelevant}. 
\begin{equation}\label{relevant}
\begin{cases}
\exists \ \underline{x}\in\partial 
\Omega,  \ \underline \alpha\in A  \text{ such that 
}\phi(\underline x)=\min \phi \text{ and }\\ 
 \text{either } \sigma^T (\underline x,\underline\alpha) Dd(\underline x)\ne 0,
\text{or }b(\underline x,\underline\alpha)\cdot Dd(\underline x) +\tr(a(\underline 
x,\underline\alpha) D^2d(\underline x))<0 .
\end{cases}
\end{equation}

\begin{theorem}%corollary}
Let \eqref{irrelevant}, \eqref{reg22s}, \eqref{compact} and \eqref{relevant} hold. 
Then the value function $v$ satisfies
\[
\forall  x\in\Omega,\qquad v(x) = \min\{\min \phi, 0\} .
\]
\end{theorem}%corollary} 
\begin{proof}  The value function $v$ is known to satisfy a dynamic programming principle (see \cite{hl}). 
From this it is possible to deduce (see \cite[Thm 4.4]{bb}, \cite{br}) that the upper semicontinuous envelope  of $v$, defined as 
$$v^*(x):=\limsup_{%\ovl\om\ni 
\overline\Omega\ni y\to x}\, v(y),\qquad x\in\ovl\om, $$ 
 is a
viscosity subsolution of the HJB equation $F[u] = 0$ in $\om$ and satisfies the 
Dirichlet boundary condition $u\leq\phi$ on $\de\om$ in viscosity sense (that 
we will recall later). 
Then, by Theorem \ref{liouville}, $v^*$ is constant in $\om$, say 
$v^*\equiv c$ in $\om$. 

By Proposition \ref{viable}, for every $x\in\Omega$ there exists an admissible control $\alpha$ such 
that $X_t^\alpha\in  \Omega $  almost surely for all $t\geq 0$. 
So, if  $\min\phi>0$, then it is immediate to deduce by the definition of $v$ that $v\equiv 0$ in $\Omega$. 

We assume now that $\min\phi\leq 0$, and we show that in this case $v\equiv \min \phi$ on $\Omega$. 
Being upper semicontinuous, $v^*$ satisfies
$v^*(x)\geq c$ for $x\in\de\om$. Since $v^*\geq v \geq \min \phi=:m$,  we 
must prove that $m\geq c$. 
%We will do it by showing that $\lim_{x\to\ox}v^*(x)=...$
Assumption \eqref{relevant} allows us to build an upper barrier %for $v$
 at the point $\underline x$, namely, % from .
%DEF DI BARR. LOCALE  {%\emupper local barrier\/}  for problem %\rec{dp.o} at $\ox\in\de\om$, i.e.
a function $W\in C^2({B}(\underline x,r)%\cap\om
)$,
with $r>0$, such that: %is an {\em upper local barrier\/} 

\noindent i) $W\geq 0$  and $LW>0$ %is a supersolution of the PDE in %\rec{dp.o}
 in
$B(\underline x,r)\cap\ovl\om$,

\noindent ii) $W(\underline x)=0$, $W(x)\geq\mu>0$ for all $x\in\ovl\om$ with 
$|x-\underline x|=r$.

\noindent %iii) $W$ is continuous at $\ox$.
For $k,\lambda>0$ set 
\[ 
W(x):=1-e^{-k(d(x)+\lambda|x-\underline x|^2)} .
\]
It is easy to compute 
%\[
 \begin{multline*}%{eqnarray*}  
\Big(-\tr(a(\underline x,\al)D^2W(\underline x))-b(\underline x,\al)\cdot 
DW(\underline x)\Big)
e^{k(d(x)+\lambda|x-\underline x|^2)}
%+c^{\alpha}(x_0)W(x_0)-f^{\alpha, \beta}(x_0)
=
\\
 -k\tr(a(\underline x,\al)D^2d(\underline x))+ k^2 
|\sigma(\underline x,\al)Dd(\underline x)|^2
-k b(\underline x,\al)\cdot Dd(\underline x)-
%\\
2k\lambda\tr a%^{\al}
(\underline x,\al).%-f^{\alpha}(x_0).
\end{multline*}%{eqnarray*} %\]
Next we choose $\al=\underline\al$ and assume first the first condition in \eqref{relevant} to hold. 
%\rec{ip.1}.
 In this
case, since the coefficients are %bounded and
 continuous and $d$ is $C^2$ close to $\partial\om$, we can get %make 
$F[W]>0$ %a supersolution of  $the PDE in %\rec{dp.o}
 in a
neighborhood of $\underline x$ by taking $k$ large enough. If, instead,
the second condition in 
\eqref{relevant}  holds, we choose %first
 $\lambda$ small %and then $\delta$ large 
 to reach the same conclusion. %\qed} 
Properties ii) of the barrier are obvious. 

%Now we
Assume by contradiction that $m<c$ and fix $m'\in(m,c)$ and $r'>0$ 
%use the continuity of $\phi$ to find $r>0$
 such that 
\[
\forall x\in B(\underline x, r')\cap \de\om,\qquad \phi(x)\leq m'<c.
\]
Now call $\rho:=\min(r,r')$ and define, for $k>0$,
\[
w(x):= k W(x) + m.
\]
Observe that $F[w] >0$ in 
$B(\underline x,\rho)\cap\ovl\om$ by the homogeneity of the operator $F$. 
Choose also 
$k$ large enough so that
\[
\forall x\in\ovl \om,\;\ |x-\underline x|=\rho,\qquad
w(x)\geq M.
\]
Next take $x_0\in \overline{B(\underline x,\rho)}\cap\ovl\om$ such that
\[
(v^*-w)(x_0)=\max_{\overline{B(\underline x,\rho)}\cap\ovl\om}(v^*-w) .
\]
There are three possible cases.

%\noindent
 1. If $x_0\in \de B(\underline x,\rho)\cap \ovl\om$ then $w(x_0)\geq M\geq 
v^*(x_0)$. 
It follows that $v^*(x)\leq w(x)$ for 
$x\in\overline{B(\underline x,\rho)}\cap\ovl\om$, which, taking $x$ in a 
neighborhood of 
$\underline x$, yields the contradiction %$\lim_{x\to\ox}v^*(x)=...$
$c \leq w(x)\leq m'$.

%\noindent
2. If $x_0\in  B(\underline x,\rho)\cap \om$ we use that $v^*$ is a subsolution 
of 
$F[u]=0$ to get $F[w](x_0)\leq 0$, a contradiction with $F[w] >0$.

3. Finally, if $x_0\in B(\underline x,\rho)\cap \de\om$, we use that 
$v^*$ is a viscosity subsolution of the boundary condition, namely, either 
$v^*(x_0)\leq \phi (x_0)$ or $F[w](x_0)\leq 0$. The latter case is impossible 
because $F[w] >0$, whereas the former contradicts $\phi(x)<c\leq v^*(x)$ in 
$B(\underline x,\rho)\cap \de\om$.

In all cases we reach a contradiction and complete the proof.
 \end{proof}
  \begin{remark}
  The conclusion of the last theorem still holds if the $C^2$ regularity of 
$\de\om$ holds only in a neighborhood of $\underline x$, provided  that  
 \eqref{irrelevant} is replaced by \eqref{d_subsol} and 
\eqref{sigma_perp}. In fact the Liouville property still holds by Proposition 
\ref{L-nonsmooth}.
\end{remark} 

\section{The ergodic HJB   equation in  invariant bounded domains}
 \label{sec:ergo}
In this Section we will 
assume  the stronger condition \eqref{invariance} on the behavior of the 
coefficients at the
boundary
 of the domain 
$\Omega$, 
as well as the strict ellipticity 
of the operator $L$ in the interior of $\Omega$ %, namely, 
 \eqref{sell}.  
We will see in Section \ref{stoch_inter} that the assumption \eqref{invariance} is related to the invariance of both $\overline \Omega$ and $\Omega$ for the control system \eqref{sys}.

\subsection{Well-posedness of the PDE} 
This subsection is dedicated to the proof of Theorem \ref{cellthm}.
We set
\begin{equation}\label{hop}H(x,p,X):= \sup_{\alpha\in A} \left(-b(x,\alpha)\cdot 
p- 
\tr (a(x,\alpha) X)-l(x,\alpha)
\right),
\end{equation} 
where $l:\overline{\Omega}\times A\to \R$ 
 is a bounded function satisfying \eqref{reg3}.
In order to find a solution to \eqref{cell}, we consider the approximated 
problems
\begin{equation}\label{l}  
\lambda u_\lambda(x)+H(x,Du_\lambda(x), D^2u_\lambda(x)))=0, \qquad x\in\Omega, 
\end{equation} 
for $\lambda>0$. 
We start with showing that this equation admits a  unique viscosity solution, 
without prescribing any boundary conditions.

We recall the following well-known a priori estimates.  
\begin{lemma}[Krylov-Safonov  estimates] \label{ks} 
Let Assumptions  \eqref{reg2}, \eqref{sell}  and \eqref{reg3} 
hold and let $u\in C(\Omega)$ be a bounded 
viscosity solution to  \eqref{l} with $\lambda>0$.
Then $u\in C^2(\Omega)$, and for every compact set 
$K\subset\Omega$ there exists $\gamma\in (0,1)$, 
depending on the space dimension $n$, the ellipticity constants in $K$ and 
$\beta$ from \eqref{reg2} such that
\[\|  u\|_{C^{1, \gamma}(K)}\leq C_K,\] 
with $C_K$ depending on $\gamma$, $K$, 
$\Omega$, $\|\sigma\|_{\infty}$, $B$ from 
\eqref{reg2}, $\|b\|_\infty$, $\|l\|_\infty$, $\|u\|_\infty$ and any 
$\overline\lambda\geq\lambda$. 
 
\end{lemma} \begin{proof} 
For the proof we refer to \cite{tr3}, Theorem 2.1 and section on further 
regularity, pag. 950. See also \cite{s}, \cite[Sec. 17.4]{gt}. 
\end{proof} 
\begin{theorem} \label{ulambda} 
Under the assumptions \eqref{reg2}, \eqref{invariance}-\eqref{reg3}, for every 
$\lambda>0$, equation 
\eqref{l} admits a unique bounded solution $u_\lambda\in C^2(\Omega)$. 
Moreover for every $\tilde x\in\Omega$ and every $K\subset\subset\Omega$, the family $( u_\lambda-u_\lambda(\tilde x) )_{\lambda\in(0,1]}$ is bounded in 
$C^{1,\gamma}(K)$, for some $\gamma\in(0,1)$. 
Finally, for all $h>0$, there exists $\delta_h>0$ such that
\begin{equation}\label{u<log}
 \forall \lambda\in(0,1],\ x\in\om_{\delta_h},\qquad 
h\log(d(x))+\min_{\om\backslash\om_{\delta_h}}u_\lambda
\leq u_\lambda(x)\leq 
-h\log(d(x))+\max_{\om\backslash\om_{\delta_h}}u_\lambda.
\end{equation}
\end{theorem}
\begin{proof}  The proof is divided into three parts.

Step 1.{ \em  Existence.}\\
The idea is to apply Perron's method to the Neumann problem for \eqref{l}, 
namely, under the boundary condition $\partial_\nu u=0$. Notice that 
the functions $\pm\|l\|_{\infty}/\lambda$ are sub and supersolutions for such
problem. However, a technical 
difficulty in the comparison principle comes from the
lack of Lipschitz-continuity of the terms $a$, $b$. 
Also, the application of Perron's method to this problem is achieved in 
\cite{barles} under the additional assumption that $\om$ is of class 
$W^{3,\infty}$. 
To overcome these difficulties one can proceed as follows.
Consider a sequence of smooth approximations 
of the $a(\cdot,\alpha)$, $b(\cdot,\alpha)$ and a sequence of smooth domains 
invading $\om$. The associated Neumann problems admit solution between 
$-\|l\|_{\infty}/\lambda$ and $\|l\|_{\infty}/\lambda$ by \cite{barles}.
Finally, using the estimates 
provided by Lemma \ref{ks} and the stability of 
viscosity solutions, one can pass to the limit along a subsequence of such 
solutions and obtain a solution $u_\lambda$ to \eqref{l}. Notice that the limit 
is only local in $\om$ 
and therefore we lose the information about the boundary behavior of 
$u_\lambda$. We only know that $u_\lambda$ is in $C^2(\om)$ and satisfies
\begin{equation}\label{b2}
\forall x\in\om,\qquad
-\frac{\|l\|_{\infty}}{\lambda}\leq u_\lambda(x)\leq 
\frac{\|l\|_{\infty}}{\lambda}.
\end{equation}

Step 2.{ \em  Uniqueness.}\\
Consider two bounded viscosity solutions $u,v\in C(\Omega)$ to \eqref{l}. Then 
they are both in 
$C^2(\Omega)$ by Lemma \ref{ks}. 
We first modify $v$ in order to obtain a supersolution blowing up 
at the boundary.
To this end, we will make use of the Lyapunov function $-\log(d(x))$ for the 
operator~$F$ defined by \eqref{F}.
Take $\delta>0$ from Proposition \ref{liap22}, associated with $M=0$, and 
consider the 
function $-\log(d(x))$ defined for $x\in\om_\delta$. Let 
$V$ be a $C^2$ extension of this function to the whole $\om$. Up to 
replacing $\delta$ with $\min(\delta,1)$, we have that $V\geq0$ 
in $\om_\delta$.
Then, for $\eps>0$, define
$$v_\eps:=v+\eps^2 V+\eps.$$
Using the fact that $v$ satisfies \eqref{l}, we find that, for $x\in\om$,
\[\begin{split}
\lambda v_\eps (x)+H(x,D v_\eps (x),D^2 v_\eps (x)) &\geq \lambda\eps^2 
V(x)+\lambda\eps\\
&\quad+\eps^2\inf_{\alpha\in A} \left(b(x,\alpha)\cdot 
DV(x)+\tr (a(x,\alpha) D^2V(x))\right)
\\
&=\eps^2(\lambda V(x)-F[-V](x))+\eps\lambda.
\end{split}\]
This last expression is positive if $x\in\om_\delta$, 
because $F[-V]<0$ there.
Otherwise, if $x\in\om\backslash\om_\delta$, it is larger than
$\eps(\lambda-C\eps)$, where $C$ is a constant only depending on $n,\lambda$ 
and the $L^\infty$ norms of $a,b,V,DV,D^2V$ in $\om\setminus\om_\delta$. 
We deduce that $v_\eps$ is a supersolution to \eqref{l} provided $\eps$ is 
smaller than some $\eps_0$.
Now, since $(v_\eps-u)(x)\to+\infty$ as $x\to\partial\om$, the function 
$v_\eps-u$ attains its minimum on $\om$ at some point $y$. By regularity we 
have  that $Dv_\eps(y)=Du(y)$ and $D^2 v_\eps(y)\geq D^2 u(y)$. 
Therefore, if $\eps\in(0,\eps_0)$, we obtain 
\[\begin{split}
0 &\leq \lambda v_\eps (y)+\sup_{\alpha\in A} \left(-b(y,\alpha)\cdot 
Dv_\eps(y)- \tr (a(y,\alpha) D^2v_\eps(y))-l(y,\alpha)\right)\\ 
&\leq \lambda(v_\eps-u)(y)+ \lambda u(y)+\sup_{\alpha\in A} 
\left(-b(y,\alpha)\cdot Du(y)- \tr 
(a(y,\alpha) D^2u(y))-l(y,\alpha)\right)\\
&=\lambda(v_\eps-u)(y).
\end{split}\]
It follows that $v_\eps\geq u$ in $\om$ and thus $v\geq u$ in $\om$ by the 
arbitrariness of $\eps\in(0,\eps_0)$.
Reversing the roles of $u,v$ we eventually derive $u\equiv v$.

Step 3.{ \em  A priori bounds.}\\ 
We start with deriving \eqref{u<log}. Fix $h>0$ and let $\delta_h$ be the 
$\delta$ given by Proposition \ref{liap2} with $M=2\|l\|_\infty/h$. Set 
$$V(x):=-h\log(d(x))+\max_{\om\backslash\om_{\delta_h}}u_\lambda.$$
It is not restrictive to assume that ${\delta_h}\leq1$, so that 
$V(x)\geq u_\lambda(x)$ if $d(x)={\delta_h}$. 
% Notice that this property remains true 
% with $u_\lambda$ replaced by the $u_\lambda^j$ used in the step 1, provided
% $n\;(>1/{\delta_h})$ is large enough.
On the other hand, there exists $\delta'\in(0,\delta_h)$ small 
enough, depending on $\lambda$, such that $V(x)>\|l\|_{\infty}/\lambda\geq 
u_\lambda(x)$
if $d(x)<\delta'$.
% the outer normal to the set $\{x\in\om\ :\ d(x)=1/n\}$ is 
% $\nu=-Dd$, whence
% $$\partial_\nu V=\frac{|Dd|^2}d>0.$$
Finally, since $u_\lambda\geq-\|l\|_{\infty}/\lambda$ and 
$F[V]>hM=2\|l\|_\infty$, 
for $x\in\om_{\delta_h}$ we get
$$\lambda V(x)+\sup_{\alpha\in A} \left(-b(x,\alpha)\cdot DV(x)- 
\tr (a(x,\alpha) D^2V(x))-l(x,\alpha)
\right)\geq\lambda\max_{\om\backslash\om_{\delta_h}}u_\lambda+F[V]-\|l\|_\infty
>0.$$
It follows from the fact that $u_\lambda$ is a (sub) solution to \eqref{l} 
that
$$\forall\delta\in(0,\delta'),\qquad
\max_{\overline{\om_{\delta_h}}\backslash\om_\delta}(u_\lambda-V)=
\max_{\partial(\overline{\om_{\delta_h}}\backslash\om_\delta)}(u_\lambda-V)\le0.
$$
Namely, $u_\lambda\leq V$ in $\om_{\delta_h}$, which is the second inequality 
in~\eqref{u<log}. The first inequality is 
obtained in analogous way, by using Proposition \ref{liap22} in place 
of \ref{liap2} and considering the subsolution
$V(x):=\min_{\om\backslash\om_{\delta_h}}u_\lambda+h\log(d(x))$.

We now fix $\tilde{x}\in\Omega$ 
and  we claim that the functions $(v_\lambda)_{\lambda\in(0,1]}$ defined by
$$v_\lambda(x):=  u_\lambda(x)- u_\lambda(\tilde{x})$$
are equibounded in any $K\subset\subset\om$. 
Assume by way of contradiction 
that there exists $K\subset\subset\om$ such that 
$$\eps_\lambda^{-1}:=\|v_\lambda\|_{L^\infty(K)}\to +\infty\qquad\text{as }
\lambda\to0^+.$$ 
Up to enlarging $K$ if needed, we can suppose that $\tilde x\in K$ and that, 
for $\delta_1$ from \eqref{u<log}, $\om\backslash\om_{\delta_1}\subset K$. The 
function  $\psi_\lambda(x):=\eps_\lambda v_\lambda(x)$ 
satisfies $\|\psi_\lambda\|_{L^\infty(K)}=1$, 
$\psi_\lambda(\tilde{x})=0$ and
\[
 \lambda \psi_\lambda+\lambda \eps_\lambda u_\lambda(\tilde{x})+\sup_{\alpha\in A} \left(-b(x,\alpha)\cdot D\psi_\lambda(x)- \tr (a(x,\alpha) D^2\psi_\lambda(x))-\eps_\lambda l(x,\alpha)
\right)=0, \qquad x\in\Omega.\]
Note that $|\lambda\eps_\lambda u_\lambda(\tilde{x})|\leq 
\eps_\lambda\|l\|_\infty\to 0$ as $\lambda\to 0$.  
Furthermore, by \eqref{u<log}, for $x\in\om\backslash K$,
\[\psi_\lambda(x)=\frac{u_\lambda(x)- u_\lambda(\tilde{x})}{\|u_\lambda- 
u_\lambda(\tilde{x})\|_{L^\infty(K)}}\leq
\frac{\max_{\om\backslash\om_{\delta_1}}u_\lambda
-\log(d(x))- u_\lambda(\tilde{x})}{\|u_\lambda- 
u_\lambda(\tilde{x})\|_{L^\infty(K)}}\leq1-\eps_\lambda\log(d(x)),\]
and
\[\psi_\lambda(x)\geq
\frac{\min_{\om\backslash\om_{\delta_1}}u_\lambda
+\log(d(x))- u_\lambda(\tilde{x})}{\|u_\lambda- 
u_\lambda(\tilde{x})\|_{L^\infty(K)}}\geq-1+\eps_\lambda\log(d(x)).\]
A first consequence of these estimates is that, in any compact subset of $\om$, 
the $\psi_\lambda$ are equibounded, whence equibounded in $C^{1,\gamma}$ 
by Lemma \ref{ks}. Using a diagonal procedure, we can then find a sequence
$\lambda\to0$ for which the $\psi_\lambda$ converge locally uniformly in $\om$
to some function $\psi\in C(\Omega)$.  By stability property of viscosity 
solutions, $\psi$  solves 
\[ \sup_{\alpha\in A} \left(-b(x,\alpha)\cdot D\psi (x)- \tr (a(x,\alpha) D^2\psi (x)) 
\right)=0, \qquad x\in\Omega.\] 
We further know that 
$\|\psi\|_{L^\infty(K)}=1$, and, for $x\in\om\backslash K$, $|\psi(x)|\leq1$
because $|\psi_\lambda(x)|\leq1-\eps_\lambda\log(d(x))$.
This means that $\psi$ attains either the global maximum $1$ or minimum $-1$ 
in $K$, and therefore it is constantly equal to $1$ or $-1$ by the Strong 
Maximum or Minimum Principle, which holds by \eqref{sell}.
% 
% This means that $\psi$ is bounded and we can therefore apply
% Theorem \ref{liouville} to infer that it is constant. 
This is impossible because $\psi(\tilde{x})=0$.

We have shown that the $v_\lambda$ are equibounded in any $K\subset\subset\om$. 
Since they satisfy \eqref{l} with $l$ replaced by $l+\lambda u_\lambda 
(\tilde{x})$, and $|\lambda u_\lambda(\tilde{x})|\leq\|l\|_\infty$, 
Lemma \ref{ks} eventually yields that they are equibounded in 
$C^{1,\gamma}(K)$.

\end{proof} 

We are now in the position to prove our main result.

\begin{proof}[Proof of Theorem \ref{cellthm}]
First of all we prove that there exists $c\in \R$ such that  \eqref{cell} admits a solution. 
Consider the solutions $(u_\lambda)_{\lambda\in(0,1]}$ to \eqref{l}. Fix 
$\tilde{x}\in\Omega$. By \eqref{b2},  $\lambda u_\lambda(\tilde{x})$ 
converges (up to subsequences) to some value $-c$ as $\lambda\to0$. 
Define $v_\lambda=   u_\lambda - u_\lambda (\tilde{x})$.   Theorem 
\ref{ulambda} gives that 
the $v_\lambda$ are equibounded in $C^{1,\gamma}(K)$, for 
any $K\subset\subset\Omega$. Thus, using a diagonalization 
procedure, we can extract a subsequence of $v_\lambda$ converging locally 
uniformly to $\chi\in C(\Omega)$, which, by stability,  is a viscosity 
solution to \eqref{cell}. We know from Lemma~\ref{ks} that $\chi\in 
C^2(\Omega)$. Moreover, $\chi$ satisfies the same bounds \eqref{u<log} as the 
$u_\lambda$ (and the $v_\lambda$). Using the fact that such bounds hold true 
for arbitrary $h>0$, we eventually find that $\chi$ fulfils \eqref{chi}.
Uniqueness of $c$ and $\chi$ (up to constants) follows by the following stronger uniqueness result.
\end{proof} 
\begin{proposition}\label{uniprop} Let $\chi_1$ and $\chi_2$ be viscosity solutions of \eqref{cell} with, respectively, $c=c_1$ and $c=c_2$. Assume moreover there exist $\kappa>0$ such that 
\[
\lim_{x\to\partial\Omega} \chi_1(x)d(x)^{\kappa }=0=\lim_{x\to\partial\Omega} \chi_2(x)d(x)^{\kappa}. 
\]
 Then %necessarily
  $c_1=c_2$
  and there exists a constant $k\in\R$ such that  $\chi_1\equiv \chi_2 +k$. 
 \end{proposition} 
\begin{proof} 
First of all, observe that by Lemma \ref{ks}, $\chi_1, \chi_2\in C^2(\Omega)$. Without loss of generality we can assume $c_1\geq c_2$. So,
\begin{equation}\label{meno}
\sup_{\alpha\in A} \left(-b(x,\alpha)\cdot D(\chi_1-\chi_2)- \tr (a(x,\alpha) 
D^2(\chi_1-\chi_2)) 
\right) \geq c_1-c_2\geq 0. 	\end{equation} By Theorem \ref{liouville2} and Remark \ref{remcrescita}, $\chi_1-\chi_2$ is  a constant, and therefore $c_1=c_2$. 
\end{proof} 
\begin{remark} \label{regularityrem} 
If we weaken the H\"older regularity  \eqref{reg3} on the coefficients $b,l$ of equation \eqref{l} to the  uniform continuity as stated in \eqref{reg22}, Krylov Safonov estimates stated in Lemma \ref{ks} still hold, but we just expect that any bounded continuous viscosity solution to \eqref{l} is $C^{1,\alpha}(\om)$ for every $\alpha\in (0,1)$, not  in $C^2(\om)$. 

This implies that in Theorem \ref{ulambda} and in Theorem \ref{cellthm}  the solutions $u_\lambda$ to \eqref{l} and $\chi$  to
\eqref{cell} are of class $C^{1,\alpha}(\om)$ for every $\alpha\in (0,1)$. So, in the  step Uniqueness of the proof of Theorem \ref{ulambda} the argument has to be modified by using the by now standard  doubling variables argument in the theory of viscosity solutions (see e.g. \cite{cil}). Moreover, also in the proof of Theorem \ref{cellthm}, the argument to prove that $\chi_1-\chi_2$ solves \eqref{meno} has to be modified appropriately.
 
\end{remark}

\subsection{A stochastic control 
 interpretation}\label{stoch_inter}

We show an application of the previous results to an ergodic control problem in 
$\Omega$ with state constraints. 
Throughout this subsection, we will assume 
the stronger regularity assumptions on the coefficients \eqref{reg22s}
to hold. It is known from \cite{bg, bj} that, under the 
assumption~\eqref{invariance} on the behavior of the coefficients near
$\partial\Omega$, the set $\overline\Omega$ is {\em invariant} for the control 
system~\eqref{sys} in the following sense: for every $x\in \overline\Omega$ and 
any admissible control $\alpha_\cdot\in{\mathcal A}$, the trajectory of \eqref{sys} 
satisfies $X_t^{\alpha_\cdot}\in \overline{ \Omega}$  almost surely for all $t\geq 0$.

Therefore no restrictions on the controls are needed to keep the system forever 
in $\overline\Omega$, and we can define the value function of the infinite 
horizon discounted control problem with {\em state constraint $\overline\Omega$} 
\begin{equation}
\label{vl}
v_\lambda (x):=  \inf_{\alpha_\cdot\in\mathcal{A}} \E 
\left[\int_0^\infty e^{-\lambda t} l(X_t^{\alpha_\cdot},\alpha_t) dt \right] , \quad 
x\in\overline\Omega
\end{equation} where $l:\overline{\Omega}\times A\to \R$ 
 is a bounded function satisfying \eqref{reg22} and $\lambda>0$.
The next result states that $v_\lambda$ is the solution of the PDE \eqref{l}.
\begin{proposition} 
Let \eqref{invariance} and \eqref{reg22s} 
hold. Then $v_\lambda$ is continuous in $\overline\Omega$ and it is a 
viscosity solution of \eqref{l}. 

If, in addition, \eqref{sell} holds, then $v_\lambda (x)=u_\lambda (x)$ for all 
$x\in\Omega$, where $u_\lambda$ is the smooth solution of \eqref{l} given by 
Theorem \ref{ulambda}.
  \end{proposition}
  \begin{proof}
  Fix $\epsilon >0$ and $T$ large enough so that $\int_T^\infty e^{-\lambda t} l(X_t^{\alpha_\cdot},\alpha_t) dt < \epsilon$ for all $x$ and $\alpha_\cdot$. Pick $x, y\in\overline\Omega$ and a control $\alpha_\cdot\in {\mathcal A}$ $\epsilon-$optimal for the initial point $y$.
  Then, denoting with $Y_t^{\alpha_\cdot}$ the trajectory starting form $y$ and using such control, 
  \[
  v_\lambda (x) - v_\lambda (y)\leq \E\left[\int_0^T e^{-\lambda t} |l(X_t^{\alpha_\cdot},\alpha_t) - l(Y_t^{\alpha_\cdot},\alpha_t) | dt\right] +3\epsilon .
  \]
 Now we use the standard estimate  $\E\left[|X_t^{\alpha_\cdot} - Y_t^{\alpha_\cdot}|\right]\leq e^{Bt}|x-y|$ and the assumption $|l(X,\alpha) - l(Y,\alpha)|\leq \omega |X-Y|$,  where the modulus $\omega$ can be assumed to be concave w.l.o.g., to get 
  \[
  v_\lambda (x) - v_\lambda (y)\leq  \int_0^T e^{-\lambda t} \omega (e^{Bt}|x-y|) dt +3\epsilon
  \]
 and the right hand side can be made smaller than $4\epsilon$ by choosing $|x-y|$ small enough. Then the continuity of $v_\lambda$ is obtained by repeating the argument with the roles of $x$ and $y$ reversed. Once this is established the Dynamic Programming Principle is a standard result, as well as deducing from it that $v_\lambda$ solves the equation in $\Omega$ in viscosity sense, see, e.g., \cite{FS}. 
 
The last statement follows from   Theorem \ref{ulambda}.
    \end{proof}

Next we show that % We start showing that
 under the conditions \eqref{reg22s} and \eqref{invariance} of this section also the open set $\Omega$ is  %{\em
  invariant for the control system \eqref{sys}, %a stronger version of
  in analogy with Proposition \ref{viable}. % st hold.

\begin{proposition}  \label{prop:invariance}
Assume \eqref{invariance} and \eqref{reg22s}. %, ? \eqref{compact}? hold. 
 Then, for every $x\in\Omega$ and every   admissible control $\alpha_\cdot$, %we get that  
 $X_t^{\alpha_\cdot}\in \Omega$  almost surely for all $t\geq 0$, i.e. $\tau^{\alpha_\cdot}_x=+\infty$ almost surely, where $\tau^{\alpha_\cdot}_x$ is defined in \eqref{ex}. 
  \end{proposition}
\begin{proof}  The proof follows the same arguments as in the proof of Proposition \ref{viable}. 
Let $U$ be a $C^2$ extension of the function $ \log(d(x))$  to the whole $\om$, which coincides with $\log d(x)$ in a neighborhood of $\partial\Omega$. We can assume that $U\leq - 1$ 
in $\om$ and moreover, by Proposition \ref{liap3},   there exists a constant $C\geq 0$ such that 
\[F[U](x)\leq C\qquad x\in\Omega.\]

By suboptimality principles for viscosity solutions (see \cite {c}), we get
for every $\delta>0$, every $x\in\Omega\setminus\Omega_\delta$ and every 
$t\geq 0$, recalling that $V\leq -1$,
\[ 
U(x)\leq \inf_{\alpha_\cdot\in\mathcal{A}} \E 
\left[U(X_{\tau^{\delta,\alpha_\cdot}_x\wedge 
t}^{\alpha_\cdot})+C(\tau^{\delta,\alpha_\cdot}_x\wedge t)\right],\] 
and then   for all $x\in \Omega$, and all $\delta>0$ \[ 
\sup_{\alpha_\cdot\in\mathcal{A}} \P \left(\omega\ |\ 
\tau^{\delta,\alpha_\cdot}_x(\omega)\leq t\right)\leq \frac{Ct-U(x)}{-\log\delta}.\]
So,    for every $t\geq 0$, 
\[ \sup_{\alpha_\cdot\in\mathcal{A}}  \P \left(\omega\ |\ \tau^{ \alpha_\cdot}_x(\omega)\leq 
t\right)=0. \]
From this we conclude as in the proof of Proposition \ref{viable} that $\P 
\left(\omega\ |\ \tau^{ \alpha_\cdot}_x(\omega)<+\infty\right)=0$ for every 
$\alpha_\cdot\in\mathcal{A}$. 
  \end{proof}  
This result allows us to interpret $u_\lambda(x)=v_\lambda(x)$ with 
$x\in\Omega$ also as the value of the discounted infinite horizon problem with 
{\em state constraint the open set $\Omega$}. We can also give a representation 
of $u_\lambda$ that is more consistent with the method of construction by means 
of Neumann boundary conditions used in Theorem \ref{ulambda}. In fact, the 
condition $\partial_\nu u=0$ on $\partial\Omega$ is related with the optimal 
control of systems for the state equation
\begin{equation}
\label{sys2}
\begin{cases} dX_t^{\alpha_\cdot}= b(X_t^{\alpha_\cdot}, \alpha_t)dt+\sqrt{2}\sigma(X_t^{\alpha_\cdot}, \alpha_t)dW_t-\nu(X_t^{\alpha_\cdot})dk_t ,\quad X^{\alpha_\cdot}_0=x\in\overline{\Omega}\\
k_t=\int_0^t %\mathbb
\mathbbm{1}_{\partial\Omega}(X_s^{\alpha_\cdot}) dk_s \quad\text{is 
nondecreasing},
\end{cases} 
\end{equation} 
where $X_t^{\alpha_\cdot}$ and $k_t$ are adapted continuous processes, % and
 $\nu$ is the unit outward normal to $\partial\Omega$, and 
$\mathbbm{1}_{\partial\Omega}$ is the indicator function of ${\partial\Omega}$.
This is a controlled diffusion process {\em with normal reflection} at the 
boundary (see, e.g., \cite{al}). By Proposition \ref{prop:invariance}, if 
$x\in\Omega$, the trajectory of \eqref{sys2} corresponding to a given control 
$\alpha_\cdot\in\mathcal{A}$ coincides a.s.~with the trajectory of \eqref{sys} 
associated with the same control. Therefore the solution $u_\lambda$ of the PDE 
\eqref{l}
%(x)=v_\lambda(x)$ with $x\in\Omega$ also as
is also the value function of the discounted infinite horizon problem for the system \eqref{sys2} with trajectories reflected at the boundary.

In conclusion we obtain %can give
  a stochastic representation formula for the solution pair $c, \chi$ of the ergodic Bellman equation \eqref{cell}.
\begin{corollary} Let \eqref{invariance}, \eqref{sell}, \eqref{reg22s}, %? \eqref{compact}? 
hold and let $v_\lambda$ be defined by \eqref{vl} with $X_t^{\alpha_\cdot}$ 
solving either \eqref{sys} or \eqref{sys2}. Then the constant $c$ of Theorem 
\ref{cellthm} satisfies
\begin{equation}
\label{formula}
c=\lim_{\lambda\to 0+} \lambda v_\lambda(x) \qquad\forall x\in\Omega,
\end{equation} 
and for any $\tilde x\in\Omega$, a %viscosity
 solution of \eqref{cell} corresponding to $c$ is
 \begin{equation*}
\chi(x)=\lim_{\lambda\to 0+} \left(v_\lambda(x) - v_\lambda(\tilde x)\right) ,  \qquad %\forall
 x\in\Omega,
\end{equation*} 
where  the convergence is locally uniform in $\Omega$ in both limits.
\end{corollary}
The formula \eqref{formula} shows a connection with ergodic control, since the limits 
$$\lim_{\lambda\to 0+} \lambda\E 
\left[\int_0^\infty e^{-\lambda t} l(X_t^{\alpha_\cdot},\alpha_t) dt \right] ,
%$ and $
\qquad\lim_{T\to +\infty} \frac 1T %T^{-1}
\E 
\left[\int_0^T l(X_t^{\alpha_\cdot},\alpha_t) dt \right]
$$ coincide if either one exists, by classical Abelian-Tauberian theorems.

\end{document}